\documentclass[12pt, a4paper]{article}
\usepackage{amsmath,amssymb}
\usepackage{bm}
\bmdefine{\eee}{e}
\bmdefine{\sss}{s}
\bmdefine{\xxx}{x}
\bmdefine{\yyy}{y}
\bmdefine{\zzz}{z}

\newcommand{\CCC}{\mathbb{C}}
\newcommand{\ZZZ}{\mathbb{Z}}

\newcommand{\PPP}{\mathbb{P}}
\newcommand{\KKK}{\mathbb{K}}
\newcommand{\SSS}{\mathbb{S}}

\renewcommand{\KKK}{{K}}

\newcommand{\define}{\mathrel{:=}}
\newcommand{\definebycond}{\stackrel{\mathrm{def}}{\iff}}

\newcommand{\gor}{Gorenstein}

\newcommand{\lm}{{\rm lm}}

\newcommand{\glin}{{\mathrm{GL}}}
\newcommand{\slin}{{\mathrm{SL}}}

\newcommand{\supp}{{\mathrm{supp}}}
\newcommand{\isupp}{{\mathrm{isupp}}}

\newcommand{\transpose}{^\top}
\newcommand{\sym}{{\mathrm{Sym}}}

\newcommand{\kgammagammaprime}{K[\Gamma(X,Y)]\cap K[\Gamma'
\begin{pmatrix}X\\ Y\end{pmatrix}]}

\newtheorem{thm}{Theorem}[section]
\newtheorem{fact}[thm]{Fact}
\newtheorem{example}[thm]{Example}
\newtheorem{lemma}[thm]{Lemma}
\newtheorem{cor}[thm]{Corollary}
\newtheorem{definition}[thm]{Definition}
\newtheorem{prop}[thm]{Proposition}
\newtheorem{remark}[thm]{Remark}

\makeatletter
\newcommand{\mysloppy}{\tolerance 9999 \hfuzz .5\p@ \vfuzz .5\p@}
\makeatother
\newcommand{\bigzerou}{\smash{\lower1.7ex\hbox{\bg 0}}}

\newcommand{\bigastu}{\smash{\lower1.7ex\hbox{\bg *}}}

\newcommand{\refeq}[1]{(\ref{#1})}

\numberwithin{equation}{section}

\newcommand{\mylabel}[1]{\label{#1}{\tt #1}}
\let\mylabel\label

\title{%
Action of special linear groups to the tensor of indeterminates,
classical invariants of binary forms and hyperdeterminant%
\footnote{%
The author is supported partially by 
JSPS KAKENHI Grant Number 24540040.}%
}
\author{Mitsuhiro MIYAZAKI\footnote{Kyoto University of Education, \tt g53448@kyokyo-u.ac.jp}
}

\date{}

\begin{document}
\mysloppy

\maketitle

\begin{abstract}
In this paper, we study the ring of invariants under the action of
$\slin(m,K)\times\slin(n,K)$ and 
$\slin(m,K)\times\slin(n,K)\times\slin(2,K)$
on the 3-dimensional array of indeterminates of form $m\times n\times 2$,
where $K$ is an infinite field.
And we show that if $m=n\geq 2$, then the ring of 
$\slin(n,K)\times\slin(n,K)$-invariants is generated by $n+1$
algebraically independent elements over $K$ and the action of
$\slin(2,K)$ on that ring is identical with the one defined in the
classical invariant theory of binary forms.
We also reveal the ring of $\slin(m,K)\times\slin(n,K)$-invariants and
$\slin(m,K)\times\slin(n,K)\times\slin(2,K)$-invariants completely
in the case where $m\neq n$.
\end{abstract}

\section{Introduction}

High-dimensional array data analysis is now rapidly developing and
being successfully applied in various fields.
A high-dimensional array datum is called a tensor in those 
communities.
To be precise, a $d$-dimensional array datum $(a_{i_1 i_2 \cdots i_d})_{1\leq i_j\leq m_j}$
is called a $d$-tensor or an $m_1\times\cdots\times m_d$-tensor.

A $2$-tensor is no other than a matrix.
For a matrix of indeterminates, that is, a matrix whose entries are independent
indeterminates, there are many results about the action of 
various subgroups of the general linear group and the rings of invariants under
this action.

To be precise, let $X=(X_{ij})$ be an $m\times n$ matrix of indeterminates,
that is, a matrix whose entries are independent indeterminates,
and $G$ a subgroup of $\glin(m,\KKK)$, where $\KKK$ is an infinite field.
For $P\in G$, one defines the action of $P$ on $\KKK[X]=\KKK[X_{ij}\mid 1\leq i\leq m$,
$1\leq j\leq n]$ by the $\KKK$ algebra homomorphism sending $X$ to $P\transpose X$.
Let us state this another words.
Let $V_1$ and $V_2$ be vector spaces over $\KKK$ with dimensions $m$ and $n$
respectively.
Since $G$ acts on $V_1$ linearly, one can extend it to the action of $G$
on $\sym(V_1\otimes V_2)$, the symmetric algebra of $V_1\otimes V_2$
over $\KKK$.

In this paper, we first consider the action of the product of two special linear groups
on a $3$-tensor of indeterminates, that is, a $3$-tensor whose entries are
independent indeterminates.
Let $T=(T_{ijk})_{1\leq i\leq m, 1\leq j\leq n, 1\leq k\leq 2}$ be an
$m\times n\times 2$-tensor of indeterminates.
Set $X_k=(T_{ijk})_{1\leq i\leq m, 1\leq j\leq n}$ for $k=1$, $2$.
Then $\slin(m,\KKK)\times \slin(n,\KKK)$ acts on the polynomial ring
$\KKK[T]=\KKK[T_{ijk}\mid 1\leq i\leq m$, $1\leq j\leq n$, $1\leq k\leq 2]$
by the $\KKK$-algebra homomorphism sending $X_k$ to 
$P\transpose X_kQ$ for $k=1$, $2$, where $(P,Q)\in \slin(m,\KKK)\times \slin(n,\KKK)$.

To state it another words, let $V_1$, $V_2$ and $V_3$ be $\KKK$-vector spaces
of dimensions $m$, $n$ and $2$ respectively.
Then the natural actions of $\slin(m,\KKK)$ on $V_1$ and $\slin(n,\KKK)$ on $V_2$
induces an action of $\slin(m,\KKK)\times \slin(n,\KKK)$ on $\sym(V_1\otimes V_2\otimes V_3)$.
And we show the following facts.
(1) If $m=n\geq 2$, then $K[T]^{\slin(n,K)\times\slin(n,K)}$
is generated by $n+1$ algebraically independent elements over $K$.
(2) If $n=m+\gcd(m,n)$, then $K[T]^{\slin(m,K)\times \slin(n,K)}$ is generated
by one element over $K$.
(3) If $m<n$ and $n\neq m+\gcd(m,n)$, then $K[T]^{\slin(m,K)\times\slin(n,K)}=K$.
(See Theorems \ref{thm:m=n} and \ref{thm:m<n}.)

Next we 
consider the action of 
$\slin(m,K)\times \slin(n,K)\times \slin(2,K)$ on $K[T]$,
in particular, the action of $\slin(2,K)$ on $K[T]^{\slin(m,K)\times \slin(n,K)}$.
And we show, above all things,
 that in the case where $m=n$ and $K=\CCC$, this action of $\slin(2,\CCC)$ on 
$\sym(V_1\otimes V_2\otimes V_3)^{\slin(n,\CCC)\times \slin(n,\CCC)}
=\CCC[T]^{\slin(n,\CCC)\times\slin(n,\CCC)}$ coincides with the action
of classical invariant theory of binary forms.
See Theorem \ref{thm:id to cl}.
We also make a remark on hyperdeterminant defined by 
Gelfand, Kapranov and Zelvinsky \cite{gkz1}.
See Corollary \ref{cor:nec lm}.
Since the theory of classical invariants dates back to nineteenth century,
using the accumulated results on classical invariants of binary forms
(\cite{sf}, \cite{hil}, \cite{dol}, \cite{shi}, \cite{dl}, \cite{bp1} and \cite{bp2})
and combining the results of this paper, one obtains much information about
$\slin(n,\CCC)\times\slin(n,\CCC)\times\slin(2,\CCC)$-invariant polynomials in
$\{T_{ijk}\}_{1\leq i\leq n,1\leq j\leq n,1\leq k\leq 2}$ and hyperdeterminant
of format $n-1$, $n-1$, 1.

The organization of this paper is as follows.
After establishing notation and recalling basic facts in Section \ref{sec:prelim},
we study in Section \ref{sec:sl sl} the invariants in $K[T]$ under
the action of $\slin(m,K)\times \slin(n,K)$ stated above.
In Section \ref{sec:sl sl sl}, we study the invariants in $K[T]$ under
the action of $\slin(m,K)\times \slin(n,K)\times\slin(2,K)$.
And show that the following facts.
(1) If $m\neq n$, then 
$K[T]^{\slin(m,K)\times\slin(n,K)\times\slin(2,K)}=
K[T]^{\slin(m,K)\times\slin(n,K)}$.
(2) If $m=n$ and $K=\CCC$, then the action of $\slin(2,\CCC)$ on
$\CCC[T]^{\slin(m,\CCC)\times\slin(n,\CCC)}$ is isomorphic to the action of
$\slin(2,\CCC)$ considered in the classical invariant theory of binary forms.
The assumption $K=\CCC$ is made because the classical invariant theory is 
studied under this assumption.

In Section \ref{sec:hyper}, we make a brief comment on the relation between
the classical invariant of binary forms and the hyperdeterminant.
And state a necessary condition that an
$\slin(n,\CCC)\times\slin(n,\CCC)\times\slin(2,\CCC)$-invariant is the
hyperdeterminant.
See Corollary \ref{cor:nec lm}.

The author would like to express his hearty thanks to
Professor Mitsuyasu Hashimoto who kindly informed the author that the action
of $\slin(2,K)$ on $K[T]^{\slin(n,K)\times \slin(n,K)}$ is identical with
the action of classical invarinat theory of binary forms.

\section{Preliminaries}
\mylabel{sec:prelim}

In this paper, all rings and algebras are commutative with identity element
and $K$ will always denote an infinite field.
The characteristic of $K$ is arbitrary except specified.

\begin{notation}
\begin{enumerate}
\item
Let $R$ be a ring and $\{a_{i_1 i_2 \cdots i_d}\}_{1\leq i_j\leq m_j}$
be a family of elements in $R$.
A $d$-dimensional array datum $T=(a_{i_1 i_2 \cdots i_d})_{1\leq i_j\leq m_j}$
is called a $d$-tensor or an $m_1\times\cdots\times m_d$-tensor
and each $a_{i_1 i_2 \cdots i_d}$ is called an entry of $T$.
\item
Let $M$ be an $m\times n$ matrix with entries in a ring $R$.
We denote by $\Gamma(M)$ the set of $m$-minors of $M$
(which may be empty) and by $\Gamma'(M)$ the set of 
$n$-minors of $M$.
\item
For a subring $S$ of $R$ and a tensor or a matrix $T$,
the subring generated by $S$ and the entries of $T$ 
is denoted by $S[T]$.
\item
A matrix whose entries are independent indeterminates is
called a matrix of indeterminates.
A tensor of indeterminates is defined similarly.
\item
We denote by $\lm(f)$ the leading monomial of a polynomial $f$ in
a polynomial ring with monomial order.
\item
We denote the transposed matrix of a matrix $M$ by $M\transpose$.
\item
We denote the cardinality of a set $\SSS$ by $|\SSS|$.
\item
For a vector space $V$ over $K$, we denote by $\sym(V)$ the symmetric algebra
of $V$ over $K$.
\item
For a monomial $g=X_1^{a_1}X_2^{a_2}\cdots X_u^{a_u}$ with 
indeterminates $X_1$, \ldots, $X_u$, we set $\supp(g)\define\{X_k\mid a_k>0\}$.
\item
For a positive integer $n$, we set $[n]\define\{1$, $2$, \ldots, $n\}$.
\end{enumerate}
\end{notation}

\begin{definition}\rm
Let $m$ and $n$ be positive integers.
We define
$$
\Gamma(m\times n)\define\{
[a_1, a_2,\ldots, a_m]\mid
1\leq a_1<a_2<\cdots<a_m\leq n,
a_i\in\ZZZ\}
$$
and
$$
\Gamma'(m\times n)\define\{
\langle a_1, a_2,\ldots, a_n\rangle\mid
1\leq a_1<a_2<\cdots<a_n\leq m,
a_i\in\ZZZ\}.
$$
And we define the order on $\Gamma(m\times n)$ by
$$
[a_1,\ldots, a_m]\leq [b_1,\ldots, b_m]\definebycond a_i\leq b_i\mbox{ for }
1\leq i\leq m.
$$
The order on $\Gamma'(m\times n)$ is defined similarly.
\end{definition}

We use the terminology and basic facts on 
algebras with straightening law (ASL for short) freely.
(See \cite{bv}, \cite{bh} or \cite{dep2}. Note that in
\cite{bh} and \cite{dep2}, an ASL is called an ordinal
Hodge algebra.)

\begin{definition}\rm
For an $m\times n$ matrix $M=(m_{ij})$ with entries in a
ring $R$ and $[a_1$, \ldots, $a_m]\in\Gamma(m\times n)$,
we set
$$
[a_1,\ldots, a_m]_M\define \det(m_{i a_j})
$$
and for a standard monomial 
$\mu=\prod_{k=1}^l[a_{k1}$, \ldots, $a_{km}]$ on $\Gamma(m\times n)$,
we set
$$
\mu_M\define
\prod_{k=1}^l[a_{k1}, \ldots, a_{km}]_M.
$$
$\langle b_1$, \ldots, $b_n\rangle_M$ 
and $\nu_M$ for
$\langle b_1$, \ldots, $b_n\rangle\in\Gamma'(m\times n)$ and a standard
monomial on $\Gamma'(m\times n)$ is defined similarly.
\end{definition}

Now let $X=(X_{ij})$ be an $m\times n$ matrix of indeterminates.
First we recall the following:

\begin{fact}
$K[\Gamma(X)]$ is an ASL on $\Gamma(m\times n)$ over $K$
by the embedding $[a_1$, \ldots, $a_m]\mapsto[a_1$, \ldots, $a_m]_X$.
\end{fact}
Suppose that a diagonal monomial order is defined on $K[X]$.
Then the leading monomial of $[a_1$, \ldots, $a_m]_X$
is $X_{1 a_1}X_{2 a_2} \cdots X_{m a_m}$.
Moreover, the following result holds.

\begin{lemma}[{c.f.\ \cite[Lemma 3.1.8]{stu}}]
\mylabel{lem:stu 3.1.8}
If $\mu$ and $\mu'$ are standard monomials on $\Gamma(m\times n)$
with $\mu\neq\mu'$, then $\lm(\mu_X)\neq\lm(\mu'_X)$.
In particular, if
$$
f=\sum_{k} c_k\mu_k
$$
is the standard representation of $f\in K[\Gamma(X)]$,
there is a unique $k$ such that $\lm(f)=\lm((\mu_k)_X)$.
\end{lemma}

We define the action of $\slin(m,K)$ on $K[X]$ by 
sending $X_{ij}$ to the $(i,j)$-entry of $P\transpose X$
for any $i$ and $j$, where $P$ is an element of 
$\slin(m,K)$.
Note that this action coincides with the action on
$\sym(K^m\otimes K^n)$ induced by the natural action
of $\slin(m,K)$ on $K^m$ under the natural isomorphism
$\sym(K^m\otimes K^n)\simeq K[X]$ such that 
$X_{ij}\leftrightarrow\eee_i^{(1)}\otimes \eee_j^{(2)}$,
where $\eee_1^{(1)}$, \ldots, $\eee_m^{(1)}$ and
$\eee_1^{(2)}$, \ldots, $\eee_n^{(2)}$ are the canonical bases of
$K^m$ and $K^n$ respectively.

Under this action, the following fact is known.

\begin{thm}[{see e.g. \cite[(7.4) Proposition and (7.7) Corollary]{bv}}]
\mylabel{thm:inv sl}
$K[X]^{\slin(m,K)}=K[\Gamma(X)]$.
\end{thm}

\section{The ring of $\slin(m,K)\times \slin(n,K)$-invariants}
\mylabel{sec:sl sl}

In this section, we consider the action of $\slin(m,\KKK)\times \slin(n,\KKK)$ on 
$\KKK[T]=\KKK[T_{ijk}\mid 1\leq i\leq m$, $1\leq j\leq n$, $1\leq k\leq 2]$,
where $T=(T_{ijk})$ is an $m\times n\times 2$-tensor of indeterminates
and study the ring of indeterminates under this action.
Set $X=(T_{ij1})$ and $Y=(T_{ij2})$.
Also set $G=\slin(m,\KKK)\times \slin(n,\KKK)$.
We define the action of $G$ by
$$
(P,Q)\cdot X=P\transpose XQ\quad\mbox{and}\quad
(P,Q)\cdot Y=P\transpose YQ
$$
for $(P,Q)\in G$.
Note that this action coincides with the action on
$\sym(K^m\otimes K^n\otimes K^2)$ induced by the natural actions
of $\slin(m,K)$ on $K^m$ and $\slin(n,K)$ on $K^n$ under the natural isomorphism
$\sym(K^m\otimes K^n\otimes K^2)\simeq K[T]$ with 
$T_{ijk}\leftrightarrow\eee_i^{(1)}\otimes \eee_j^{(2)}\otimes \eee_k^{(3)}$,
where $\eee_1^{(1)}$, \ldots, $\eee_m^{(1)}$;
$\eee_1^{(2)}$, \ldots, $\eee_n^{(2)}$ and 
$\eee_1^{(3)}$, $\eee_2^{(3)}$ are the canonical bases of
$K^m$, $K^n$ and $K^2$ respectively.

By Theorem \ref{thm:inv sl} and symmetry, it is clear that the 
following fact holds.

\begin{lemma}
\mylabel{lem:kgammagammaprime}
$K[T]^{G}=\kgammagammaprime$.
\end{lemma}

By symmetry, we may assume that $m\leq n$.
If $n>2m$, then $
K[\Gamma'\begin{pmatrix}X\\ Y\end{pmatrix}]=K$,
so in the following, we also assume that $n\leq 2m$.

We introduce the degree lexicographic monomial order on $K[T]$
with
$T_{111}>T_{211}>T_{311}>\cdots>T_{m11}>T_{121}>\cdots>T_{m21}>T_{131}>
\cdots>T_{mn1}>T_{112}>T_{212}>\cdots>T_{m12}>T_{122}>\cdots>T_{mn2}$.
And we set
$Z=(X,Y)=(Z_{ij})$ and $W=\begin{pmatrix}X\\ Y\end{pmatrix}=(W_{ij})$.
Note the monomial order we defined is diagonal both for $Z$ and $W$.
So we see the following:

\begin{lemma}
\mylabel{lem:lm one elem}
For an element $[a_1$, \ldots, $a_m]\in\Gamma(m\times 2n)$,
$\lm([a_1$, \ldots, $a_m]_Z)=Z_{1a_1}Z_{2a_2}\cdots Z_{m a_m}$
and for $\langle b_1$, \ldots, $b_n\rangle\in\Gamma'(2m\times n)$,
$\lm(\langle b_1$, \ldots, $b_n\rangle_W)=W_{b_11}W_{b_22}\cdots W_{b_nn}$.
\end{lemma}

Let $f$ be a non-zero element of $\kgammagammaprime$. 
Since $f\in K[\Gamma(X,Y)]$, by Lemmas \ref{lem:stu 3.1.8} and 
\ref{lem:lm one elem}, we see that
$\supp(\lm(f))\cap\{Z_{ij}\mid j<i$ or $2n-j<m-i\}=\emptyset$.
In other words,
$$
\supp(\lm(f))\cap(\{T_{ij1}\mid j<i\}\cup\{T_{ij2} \mid n-j<m-i\})=\emptyset.
$$
By symmetry, we also see that
$$
\supp(\lm(f))\cap(\{T_{ij1}\mid i<j\}\cup\{T_{ij2} \mid m-i<n-j\})=\emptyset.
$$
Therefore, we see the following fact:

\begin{lemma}
\mylabel{lem:lm in}
Let $f$ be a non-zero element of $\kgammagammaprime$.
Then
$$
\supp(\lm(f))\subset\{T_{111}, T_{221}, \ldots, T_{mm1}, T_{1, n-m+1,2}, T_{2,n-m+2,2}, \ldots, T_{mn2}\}.
$$
In other words,
$\supp(\lm(f))\subset\{Z_{ij}\mid i=j$ or $m-i=2n-j\}$ and
$\supp(\lm(f))\subset\{W_{ij}\mid i=j\leq m$ or $2m-i=n-j< m\}$.
\end{lemma}
By Lemmas \ref{lem:stu 3.1.8}, \ref{lem:lm one elem} and \ref{lem:lm in} we see the following:

\begin{lemma}
\mylabel{lem:lm f}
Let $f$ be a non-zero element of $\kgammagammaprime$.
Also let
$$
f=\sum_{k}c_k\mu_k \quad\mbox{and}\quad
f=\sum_{l}d_l\nu_l
$$
be the standard representations of $f$ with respect to the ASL structure of
$K[\Gamma(X,Y)]$ on $\Gamma(m\times 2n)$ and  
$K[\Gamma'\begin{pmatrix}X\\ Y\end{pmatrix}]$ on $\Gamma'(2m\times n)$ respectively.
Suppose that $\mu\define\mu_k$ is the 
standard monomial on $\Gamma(m\times 2n)$ such that $\lm(f)=\lm(\mu_Z)$ 
and  $\nu\define\nu_l$ is the standard monomial on $\Gamma'(2m\times n)$
such that $\lm(f)=\lm(\nu_W)$, then $\mu$ and $\nu$ is the following form.
$$
\mu=\prod_{t=1}^u[1,2,\ldots,i_t,2n-m+i_t+1,2n-m+i_t+2,\ldots, 2n]
$$
and
$$
\nu=\prod_{s=1}^v\langle1,2,\ldots,j_s,2m-n+j_s+1,2m-n+j_s+2,\ldots, 2m\rangle,
$$
where $i_1$, \ldots, $i_u$ and $j_1$, \ldots, $j_v$ are integers with 
$m \geq  i_1\geq \cdots\geq i_u\geq 0$ and $m\geq j_1\geq \cdots\geq j_v\leq n-m$.
And
\begin{eqnarray*}
\lm(f)&=&\prod_{t=1}^u T_{111}\cdots T_{i_t i_t 1} T_{i_t+1, n-m+i_t+1,2}\cdots T_{mn2}\\
&=&\prod_{s=1}^v T_{111}\cdots T_{j_s j_s 1} T_{m-n+j_s+1,j_s+1,2}\cdots T_{mn2}.
\end{eqnarray*}
\end{lemma}

Now we study $K[T]^G$.
First we consider the case where $m=n$.

\begin{definition}\rm
Set
$$
\det(X+Y)=f_{n,0}+f_{n-1,1}+\cdots+f_{0,n},
$$
where $f_{k,n-k}$ is the sum of monomials whose degree with respect to 
$T_{ij1}$ and $T_{ij2}$ is $k$ and $n-k$ respectively.
\end{definition}

\begin{remark}\rm
\mylabel{rem:lm f}
Set $X=[\xxx_1$, \ldots, $\xxx_n]$ and $Y=[\yyy_1$, \ldots, $\yyy_n]$
and for $\SSS\subset[n]$, set
$$
\zzz_j^\SSS\define
\left\{\begin{array}{ll}
\xxx_j\quad&(j\in\SSS)\\
\yyy_j\quad&(j\not\in\SSS).
\end{array} \right.
$$
Then
$$
f_{k, n-k}=\sum_{\SSS\subset[n], |\SSS|=k}\det[\zzz_1^\SSS,\ldots,\zzz_n^\SSS].
$$
In particular, $\lm(f_{k,n-k})=T_{111}\cdots T_{kk1} T_{k+1,k+1,2}\cdots T_{nn2}$,
and therefore leading monomials of $f_{n,0}$, $f_{n-1,1}$, \ldots, $f_{0,n}$
are algebraically independent over $K$.
It follows that $f_{n,0}$, $f_{n-1,1}$, \ldots, $f_{0,n}$ are algebraically
independent over $K$.
\end{remark}

Next we state the following:

\begin{lemma}
\mylabel{lem:det and f}
$\KKK[\det(cX+dY)\mid c,d\in\KKK]=
\KKK[f_{n,0}, f_{n-1,1}$, \ldots, $f_{0,n}]$.
\end{lemma}
\begin{proof}
Since $\det(cX+dY)=\sum_{k=0}^n c^{n-k}d^k f_{n-k,k}$,
it is clear that
$\KKK[\det(cX+dY)\mid c,d\in\KKK]\subset
\KKK[f_{n,0}, f_{n-1,1}$, \ldots, $f_{0,n}]$.
On the other hand, since
$\det(cX+Y)=\sum_{k=0}^m c^{n-k} f_{n-k,k}$ and $\KKK$ is an infinite field,
by the argument using Vandermonde determinant, we see that
$f_{n-k,k}\in \KKK[\det(cX+Y)\mid c\in\KKK]$ for any $k$.
\end{proof}
Now we state the following:

\begin{thm}
\mylabel{thm:m=n}
Suppose $m=n$.
Then
$f_{n,0}$, $f_{n-1,1}$, \ldots, $f_{0,n}$ is a sagbi basis of $K[T]^G$.
In particular, $K[T]^G=K[f_{n,0}$, $f_{n-1,1}$, \ldots, $f_{0,n}]$
and $K[T]^G$ is isomorphic to the polynomial ring with $n+1$ variables over $K$.
\end{thm}
\begin{proof}
Clearly, for any $c$ and $d\in K$, $\det(cX+dY)$ is invariant under the
action of $G$.
Therefore,
$K[T]^G\supset K[\det(cX+dY)\mid c$, $d\in K]=K[f_{n,0}$, $f_{n-1,1}$, \ldots, $f_{0,n}]$
by Lemma \ref{lem:det and f}.
On the other hand, 
for any $f\in\kgammagammaprime$ with $f\neq 0$, we see that $\lm(f)$ is contained in 
$K[\lm(f_{n,0})$, \ldots, $\lm(f_{0,n})]$ by Lemma \ref{lem:lm f} and Remark \ref{rem:lm f}.
So the second sentence of the theorem follows by Lemma \ref{lem:kgammagammaprime}.
The third sentence follows from the basic property of sagbi basis
and Remark \ref{rem:lm f}.
\end{proof}

Next we consider the case where $m<n\leq 2m$.
Let $f$ a non-zero element of $\kgammagammaprime$.
We use the notation of Lemma \ref{lem:lm f}.

\begin{definition}
\rm
For an integer $i$ with $0\leq i\leq m$, we set
$$
[[i,m-i]]\define[1,2,\ldots, i,2n-m+i+1,2n-m+i+2,\ldots, 2n]
\in\Gamma(m\times 2n)
$$
and for an integer $j$ with $n-m\leq j\leq m$, we set
$$
\langle\langle j,n-j\rangle\rangle\define
\langle 1,2,\ldots,j,2m-n+j+1,2m-n+j+2,\ldots,2m\rangle
\in\Gamma'(2m\times n).
$$
\end{definition}

\begin{remark}
\rm
By Lemma \ref{lem:lm f}, we see that 
$\mu=\prod_{t=1}^u[[i_t,m-i_t]]$
and
$
\nu=\prod_{s=1}^v\langle\langle j_s,n-j_s\rangle\rangle
$.
\end{remark}

\begin{definition}
\rm
Set $l=n-m$.
Also set 
$$
\mu=\prod_{i=0}^m[[i,m-i]]^{c_i}
\quad\mbox{and}\quad
\nu=\prod_{j=l}^m\langle\langle j,n-j\rangle\rangle^{d_j},
$$
that is, 
$c_i=|\{t\mid i_t=i\}|$ and $d_j=|\{s\mid j_s=j\}|$.
We define 
$$
L(i,m-i)\define c_i
\quad\mbox{for $i=0$, $1$, \ldots, $m$}
$$
and
$$
L\begin{pmatrix}j\\ n-j\end{pmatrix}
\define d_j
\quad\mbox{for $j=l$, $l+1$, \ldots, $m$}.
$$
\end{definition}

By comparing the exponent of $T_{mm1}$ of
$\lm(f)=\lm(\mu_Z)=\lm(\nu_W)$, we see that
$$
L(m,0)=
L\begin{pmatrix} m\\ l\end{pmatrix}.
$$
Similarly, by comparing the exponent of $T_{m-1,m-1,1}$, we see that
$$
L(m-1,1)+L(m,0)=
L\begin{pmatrix} m-1\\ l+1\end{pmatrix}+
L\begin{pmatrix} m\\ l\end{pmatrix}.
$$
Therefore,
$$
L(m-1,1)=
L\begin{pmatrix} m-1\\ l+1\end{pmatrix}.
$$
By continuing the same argument, we see that
\begin{equation}
\mylabel{eq:1}
L(m-k,k)=
L\begin{pmatrix} m-k\\ l+k\end{pmatrix}
\quad\mbox{for $k=0$, $1$, \ldots, $m-l$.}
\end{equation}
By symmetry, we also see that
\begin{equation}
\mylabel{eq:2}
L(k,m-k)=
L\begin{pmatrix} l+k\\ m-k\end{pmatrix}
\quad\mbox{for $k=0$, $1$, \ldots, $m-l$.}
\end{equation}

Next, by comparing the exponent of $T_{l-1,l-1,1}$ of
$\lm(f)=\lm(\mu_Z)=\lm(\nu_W)$, we obtain the following equation.
\begin{eqnarray*}
&&
L(l-1,m-l+1)+L(l,m-l)+\cdots+L(m,0)\\
&=&
L\begin{pmatrix} l\\ m-l\end{pmatrix}+
L\begin{pmatrix} l+1\\ m-l-1\end{pmatrix}+\cdots+
L\begin{pmatrix} m\\ l\end{pmatrix}.
\end{eqnarray*}
Therefore we see, by equation \refeq{eq:1}, that
$$
L(l-1,m-l+1)=0.
$$
Similarly, by comparing the exponent of $T_{l-2,l-2,1}$, we see that
\begin{eqnarray*}
&&
L(l-2,m-l+2)+L(l-1,m-l+1)+\cdots+L(m,0)\\
&=&
L\begin{pmatrix} l\\ m-l\end{pmatrix}+
L\begin{pmatrix} l+1\\ m-l-1\end{pmatrix}+\cdots+
L\begin{pmatrix} m\\ l\end{pmatrix}
\end{eqnarray*}
and from this, we see that
$$
L(l-2,m-l+2)=0.
$$
By continuing this argument, we see that
\begin{equation}
\mylabel{eq:3}
L(1,m-1)=L(2,m-2)=\cdots=L(l-1,m-l+1)=0.
\end{equation}

Let $r$ be an integer with $0<r<l$.
By equations \refeq{eq:3}, \refeq{eq:2} and \refeq{eq:1},
we see that 
\begin{eqnarray*}
0=L(r,m-r)&=&L\begin{pmatrix}r+l\\ m-r\end{pmatrix}=L(r+l,m-r-l)\\
&=&L\begin{pmatrix}r+2l\\ m-r-l\end{pmatrix}=L(r+2l,m-r-2l)\\
&=&\cdots \\
&=&
L\begin{pmatrix}r+ql\\ m-r-(q-1)l\end{pmatrix}=L(r+ql,m-r-ql),
\end{eqnarray*}
where $q=\lfloor (m-r)/l\rfloor$.
By this equation and symmetry, we see the following:

\begin{lemma}
\mylabel{lem:not mult}
Let $k$ be an integer with $1\leq k\leq m$.
If $l\nmid k$, then $L(k,m-k)=L(m-k,k)=0$.
\end{lemma}

If $l\nmid m$, then by Lemma \ref{lem:not mult}, we see that $L(0,m)=0$.
Therefore,
\begin{eqnarray*}
0=L(0,m)&=&
L\begin{pmatrix}l\\ m\end{pmatrix}=L(l,m-l)\\
&=&
L\begin{pmatrix}2l\\ m-l\end{pmatrix}=L(2l,m-2l)\\
&=&\cdots \\
&=&
L\begin{pmatrix}ql\\ m-(q-1)l\end{pmatrix}=L(ql,m-ql),
\end{eqnarray*}
where $q=\lfloor m/l\rfloor$.
By this equation and Lemma \ref{lem:not mult}, we see the following:

\begin{lemma}
\mylabel{lem:trivial case}
If $l\nmid m$, then 
$$
L(k,m-k)=0\quad
\mbox{for $k=0$, $1$, \ldots, $m$.}
$$
In particular, $f\in K$.
\end{lemma}

Set $d=\gcd(m,n)$.
Note that $l | m$ if and only if $d=l$.
Now consider the case where $l | m$.
Set $c=L(0,m)$.
Then by equations \refeq{eq:2} and \refeq{eq:1}, we see that
\begin{eqnarray*}
c=L(0,m)&=&L\begin{pmatrix}d\\ m\end{pmatrix}=L(d,m-d)\\
&=&L\begin{pmatrix}2d\\ m-d\end{pmatrix}=L(2d,m-2d)\\
&=&\cdots \\
&=&
L\begin{pmatrix}m\\ d\end{pmatrix}=L(m,0),
\end{eqnarray*}
Therefore, by Lemmas \ref{lem:lm one elem}  and \ref{lem:lm f}, we see the following:

\begin{lemma}
\mylabel{lem:lm f m<n}
Suppose $l=d$ and set $s=m/d$.
Then 
\begin{eqnarray*}
\lm(f)&=&(T_{111}T_{221}\cdots T_{dd1})^{sc}\\
&&\times (T_{d+1,d+1,1}T_{d+2,d+2,1}\cdots T_{2d,2d,1})^{(s-1)c}\\
&&\times \cdots\\
&&\times (T_{m-d+1,m-d+1,1}T_{m-d+2,m-d+2,1}\cdots T_{mm1})^c\\
&&\times (T_{1,d+1,2}T_{2,d+2,2}\cdots T_{d,2d,2})^c\\
&&\times (T_{d+1,2d+1,2}T_{d+2,2d+2,2}\cdots T_{2d,3d,2})^{2c}\\
&&\times \cdots\\
&&\times (T_{m-d+1,n-d+1,2}T_{m-d+2,n-d+2,2}\cdots T_{mn2})^{sc}.
\end{eqnarray*}
\end{lemma}

Consider the following $(mn/d)\times (mn/d)$ determinant, 
where $m/d$ blocks of $n$ columns are aligned horizontally and 
$n/d$ blocks of $m$ rows are aligned
vertically.
\begin{equation}
\mylabel{eqn:inv det}
\left|
\begin{array}{ccccc}
X\\
Y&X\\
&Y&X\\
&&Y&\ddots\\
&&&\ddots&X\\
&&&&Y
\end{array}
\right|
\end{equation}
Note that the Laplace expansion by $n/d$ blocks of $m$ rows of this determinant
shows that this determinant is 
$\slin(m,\KKK)$-invariant
and  the Laplace expansion by $m/d$ blocks of $n$ columns of the determinant
shows that this determinant
is $\slin(n,\KKK)$-invariant.
That is, the determinant \refeq{eqn:inv det} is $G$-invariant.
We also see that the leading monomial of the determinant \refeq{eqn:inv det}
is
\begin{eqnarray*}
&&(T_{111}T_{221}\cdots T_{dd1})^{s}\\
&\times & (T_{d+1,d+1,1}T_{d+2,d+2,1}\cdots T_{2d,2d,1})^{(s-1)}\\
&\times &\cdots\\
&\times &(T_{m-d+1,m-d+1,1}T_{m-d+2,m-d+2,1}\cdots T_{mm1})\\
&\times &(T_{1,d+1,2}T_{2,d+2,2}\cdots T_{d,2d,2})\\
&\times &(T_{d+1,2d+1,2}T_{d+2,2d+2,2}\cdots T_{2d,3d,2})^{2}\\
&\times &\cdots\\
&\times &(T_{m-d+1,n-d+1,2}T_{m-d+2,n-d+2,2}\cdots T_{mn2})^{s}
\end{eqnarray*}
by examining the Laplace expansion by $n/d$ blocks of $m$ rows.

Therefore, by Lemma \ref{lem:lm f m<n}, we see the following:

\begin{lemma}
\mylabel{lem:lm f is det}
$\lm(f)$ is a power of  the leading monomial of the determinant
\refeq{eqn:inv det}.
\end{lemma}
By Lemmas \ref{lem:trivial case} and \ref{lem:lm f is det},
we see the following:

\begin{thm}
\mylabel{thm:m<n}
Suppose $m<n$ and set $d=\gcd(m,n)$.
\begin{enumerate}
\item
If $n=m+d$, then the determinant \refeq{eqn:inv det} is a
sagbi basis of $\KKK[T]^G$.
In particular, $K[T]^G$ is generated
by the determinant \refeq{eqn:inv det}.
\item
If $n\neq m+d$, then the ring of $G$-invariants in $\KKK[T]$ is $\KKK$.
\end{enumerate}
\end{thm}

\section{The action of $\slin(m,K)\times \slin(n,K)\times \slin(2,K)$ and invariants of binary forms}
\mylabel{sec:sl sl sl}

Next we consider the action of $\slin(m,\KKK)\times \slin(n,\KKK)\times \slin(2,\KKK)$ on $\KKK[T]$.
We define the action of $G=\slin(m,\KKK)\times \slin(n,\KKK)$ on $\KKK[T]$ by the same way as
in the previous section and define the action of $\slin(2,\KKK)$ by
$$
\begin{pmatrix}a&b\\c&d\end{pmatrix}\cdot X
=aX+cY
\quad\mbox{and}\quad
\begin{pmatrix}a&b\\c&d\end{pmatrix}\cdot Y
=bX+dY
$$
for $\begin{pmatrix}a&b\\c&d\end{pmatrix}\in\slin(2,K)$.

\begin{remark}\rm
The action of $\slin(m,K)\times \slin(n,K)\times \slin(2,K)$
on $K[T]$ defined above and in the previous section coincides with the action on
$\sym(K^m\otimes K^n\otimes K^2)$ induced by the natural actions
of $\slin(m,K)$ on $K^m$, $\slin(n,K)$ on $K^n$ and $\slin(2,K)$ on $K^2$ 
under the natural isomorphism
$\sym(K^m\otimes K^n\otimes K^2)\simeq K[T]$ with
$T_{ijk}\leftrightarrow\eee_i^{(1)}\otimes \eee_j^{(2)}\otimes \eee_k^{(3)}$,
where $\eee_1^{(1)}$, \ldots, $\eee_m^{(1)}$;
$\eee_1^{(2)}$, \ldots, $\eee_n^{(2)}$ and 
$\eee_1^{(3)}$, $\eee_2^{(3)}$ are the canonical bases of
$K^m$, $K^n$ and $K^2$ respectively.
\end{remark}

We first consider the case where $m<n$.
First we note the following:

\begin{lemma}
Let $s$ be a positive integer.
And let $A$ be the $(s+1)m\times sn$ matix of the following form.
$$
\left[
\begin{array}{ccccc}
X\\
Y&X\\
&Y&X\\
&&Y&\ddots\\
&&&\ddots&X\\
&&&&Y
\end{array}
\right].
$$
Then for any $c\in \KKK$, there is a sequence of blockwise row and column
operations such that adding a scalar multiple of $j$-th row block to $i$-th
row block with $i<j$ and adding  a scalar multiple of $i$-th column block
to $j$-th column block with $i<j$ and transform $A$ into
$$
\left[
\begin{array}{ccccc}
X+cY\\
Y&X+cY\\
&Y&X+cY\\
&&Y&\ddots\\
&&&\ddots&X+cY\\
&&&&Y
\end{array}
\right].
$$
\end{lemma}

\begin{proof}
We prove by induction on $s$.
The case where $s=1$ is clear.
Assume $s>1$ and let $B$ be the $sm\times (s-1)n$ matrix of the following form
$$
\left[
\begin{array}{ccccc}
X\\
Y&X\\
&Y&X\\
&&Y&\ddots\\
&&&\ddots&X\\
&&&&Y
\end{array}
\right].
$$
By the induction hypothesis, there is a sequence of row and column operations
of the type stated in the lemma which transforms $B$ to
$$
\left[
\begin{array}{ccccc}
X+cY\\
Y&X+cY\\
&Y&X+cY\\
&&Y&\ddots\\
&&&\ddots&X+cY\\
&&&&Y
\end{array}
\right].
$$
By applying this sequence of row and column operations to the upper left corner of $A$,
one gets
\begin{equation}
\mylabel{eqn:half way}
\left[
\begin{array}{cccccc}
X+cY&&&&&a_1X\\
Y&X+cY&&&&a_2X\\
&Y&X+cY&&&a_3X\\
&&Y&\ddots&&\vdots\\
&&&\ddots&X+cY&a_{s-2}X\\
&&&&Y&X\\
&&&&&Y
\end{array}
\right].
\end{equation}
By adding $-a_i$ times of the $i$-th column block  to the $s$-th column block,
one can transform the matrix \refeq{eqn:half way} to the following form.
$$
\left[
\begin{array}{cccccc}
X+cY&&&&&b_1Y\\
Y&X+cY&&&&b_2Y\\
&Y&X+cY&&&b_3Y\\
&&Y&\ddots&&\vdots\\
&&&\ddots&X+cY&b_{s-2}Y\\
&&&&Y&X+b_{s-1}Y\\
&&&&&Y
\end{array}
\right]
$$
Finally, by adding $-b_i$ times of $(s+1)$-th row block to $i$-th row block for
$i=1$, \ldots, $s-2$ and $-b_{s-1}+c$ times of $(s+1)$-th row block to $s$-th
row block, one gets a desired form.
\end{proof}
As a corollary, we see the following fact.

\begin{cor}
\mylabel{cor:same det}
Set $d=\gcd(m,n)$ and assume that $n=m+d$.
Let $A$  be the $(mn/d)\times (mn/d)$ matrix of the following form.
$$
\left[
\begin{array}{ccccc}
X\\
Y&X\\
&Y&X\\
&&Y&\ddots\\
&&&\ddots&X\\
&&&&Y
\end{array}
\right]
$$
Then the determinant of the following $(mn/d)\times(mn/d)$ matrix 
$$
\left[
\begin{array}{ccccc}
X+cY\\
Y&X+cY\\
&Y&X+cY\\
&&Y&\ddots\\
&&&\ddots&X+cY\\
&&&&Y
\end{array}
\right]
$$
is $\det A$.
\end{cor}
Since $\slin(2,\KKK)$ is generated by
$\{\begin{pmatrix}1&c\\0&1\end{pmatrix}\mid c\in\KKK\}\cup
\{\begin{pmatrix}1&0\\ c&1\end{pmatrix}\mid c\in\KKK\}$,
by Corollary \ref{cor:same det} and symmetry, we see the following:

\begin{prop}
Set $d=\gcd(m,n)$ and assume that $n=m+d$.
Then the determinant \refeq{eqn:inv det} is
invariant under the action of $\slin(2,\KKK)$.
\end{prop}
Therefore, we see the following:

\begin{thm}
\mylabel{thm:abs3}
If $m<n$, then $\KKK[T]^{G\times \slin(2,\KKK)}=\KKK[T]^G$.
\end{thm}

Next we consider the case where $m=n$.

Professor Mitsuyasu Hashimoto kindly informed the author that this action of
$\slin(2,\KKK)$ on $\KKK[f_{n,0}$, $f_{n-1,1}$, \ldots, $f_{0,n}]$ is identical with
the one of the classical invariant theory of binary forms.
Here we recall the classical invariant of binary forms
(see \cite[Section 3.6]{stu} and \cite[Section 1.3]{muk}).

From now on we assume that $K=\CCC$, the complex number field.
Let $x$ and $y$   be indeterminates.
For $\begin{pmatrix}\alpha&\beta\\ \gamma&\delta\end{pmatrix}\in \slin(2,\CCC)$ 
we define $\bar x$ and $\bar y$ by
$$
\begin{pmatrix}x\\ y\end{pmatrix}=
\begin{pmatrix}\alpha&\beta\\ \gamma&\delta\end{pmatrix}
\begin{pmatrix}\bar x\\ \bar y\end{pmatrix}.
$$
Now let $\xi_0$, $\xi_1$, \ldots, $\xi_n$ be $n+1$ new indeterminates 
and let $f$ be the following form of $x$ and $y$ of degree $n$.
$$
f=\sum_{k=0}^n{n\choose k}\xi_k x^k y^{n-k}
$$
One can rewrite $f$ by using $\bar x$ and $\bar y$
by substituting $x=\alpha \bar x+\beta\bar y$ and $y=\gamma\bar x+\delta \bar y$.
$$
f=\sum_{k=0}^n{n\choose k}\xi_k(\alpha \bar x+\beta \bar y)^k (\gamma \bar x+\delta \bar y)^{n-k}
=\sum_{k=0}^n{n\choose k}\bar\xi_k \bar x^k \bar y^{n-k}
$$

\begin{definition}\rm
The action of $\slin(2,\CCC)$ on $\CCC[\xi_0$, $\xi_1$, \ldots, $\xi_n$, $x$, $y]$
is defined so that 
$\begin{pmatrix}\alpha&\beta\\ \gamma&\delta\end{pmatrix}\in \slin(2,\CCC)$ 
maps $x$ to $\bar x$, $y$ to $\bar y$ and $\xi_k$ to $\bar \xi_k$
for $k=0$, $1$, \ldots, $n$.
An element of 
$\CCC[\xi_0$, \ldots, $\xi_n$, $x$, $y]^{\slin(2,\CCC)}$ is called a covariant and
an element of 
$\CCC[\xi_0$, \ldots, $\xi_n]^{\slin(2,\CCC)}$ is called an invariant.
\end{definition}

Recall that under the action of $\slin(2,\CCC)$ on $\CCC[T]$,
$\begin{pmatrix}\alpha&\beta\\ \gamma&\delta\end{pmatrix}\in \slin(2,\CCC)$ 
maps $X$ to $\alpha X+\gamma Y$ and $Y$ to $\beta X +\delta Y$.
Set $\bar X=\alpha X+\gamma Y$ and $\bar Y=\beta X+\delta Y$.
Then since 
$\begin{pmatrix} \bar x\\ \bar y\end{pmatrix}=
\begin{pmatrix}\alpha & \beta\\ \gamma & \delta\end{pmatrix}^{-1}
\begin{pmatrix} x\\  y\end{pmatrix}$,
we see that
$$
\bar x\bar X+\bar y\bar Y=xX+yY.
$$

Set $\bar X=[\bar \xxx_1$, \ldots, $\bar \xxx_n]$ and 
$\bar Y=[\bar \yyy_1$, \ldots, $\bar \yyy_n]$
and for $\SSS\subset[n]$, set
$$
\bar \zzz_j^\SSS\define
\left\{\begin{array}{ll}
\bar \xxx_j\quad&(j\in\SSS)\\
\bar \yyy_j\quad&(j\not\in\SSS).
\end{array} \right.
$$
Also set
$$
\bar f_{k, n-k}=
\sum_{\SSS\subset[n], |\SSS|=k}
\det[\bar \zzz_1^\SSS,\ldots,\bar \zzz_n^\SSS].
$$
Since by
$\begin{pmatrix}\alpha&\beta\\ \gamma&\delta\end{pmatrix}\in \slin(2,\CCC)$,
$\xxx_k$ and $\yyy_k$ are mapped to $\bar\xxx_k$ and $\bar \yyy_k$ respectively
for $k=0$, \ldots, $n$, we see, by Remark \ref{rem:lm f}, the following:

\begin{lemma}
$f_{k,n-k}$ is mapped to $\bar f_{k,n-k}$
by the action of
$\begin{pmatrix}\alpha&\beta\\ \gamma&\delta\end{pmatrix}\in\slin(2,\CCC)$.
\end{lemma}
On the other hand, by expanding 
$\det(xX+yY)$ and 
$\det(\bar x\bar X+\bar y\bar Y)$, one obtains 
$$
\sum_{k=0}^n f_{k,n-k}  x^k  y^{n-k}=
\det (xX+yY)=
\det (\bar x\bar X+\bar y\bar Y)=
\sum_{k=0}^n\bar f_{k,n-k} \bar x^k \bar y^{n-k}.
$$
Therefore, we see the following:

\begin{thm}
\mylabel{thm:id to cl}
$\CCC[f_{n,0}$, $f_{n-1,1}$, \ldots, $f_{0,n}]$ is isomorphic to
$\CCC[\xi_0$, $\xi_1$, \ldots, $\xi_n]$ as $\slin(2,\CCC)$-modules
by the map
$f_{k,n-k}\mapsto {n\choose k}\xi_k$ 
for $k=0$, $1$, \ldots, $n$.
\end{thm}

The theory of
classical invariants of binary forms 
dates back to nineteenth century.
And is still a theme of research in progress.
See \cite{sf}, \cite{hil}, \cite{dol}, \cite{shi}, \cite{dl}, \cite{bp1} and \cite{bp2}. 
And by using these results, we obtain the information of the structure
of $\CCC[T]^{G\times \slin(2,\CCC)}$.

\begin{example}
\rm
\mylabel{ex:inv}
\begin{enumerate}
\item
The ring of invariants of 
a binary quadric is generated by 
$\xi_1^2-\xi_0\xi_2$ over $\CCC$.
Therefore, if $m=n=2$, then $\CCC[T]^{G\times\slin(2,\CCC)}$ is generated by
$f_{11}^2-4f_{02}f_{20}$
over $\CCC$.
\item
The ring of invariants of a binary cubic is generated by
$3\xi_1^2\xi_2^2-4\xi_0\xi_2^3-4\xi_1^3\xi_3+6\xi_0\xi_1\xi_2\xi_3-\xi_0^2\xi_3^2$
over $\CCC$.
Therefore, if $m=n=3$, then $\CCC[T]^{G\times \slin(2,\CCC)}$ is generated by
$f_{12}^2 f_{21}^2-4f_{03}f_{12}^3-4f_{12}^3f_{30}+18f_{03}f_{12}f_{21}f_{30}-27f_{03}^2f_{30}^2$
over $\CCC$.
\item
\mylabel{item:n=4}
The ring of invariants of a binary quartic is generated by
$3\xi_2^2-4\xi_1\xi_3+\xi_0\xi_4$ and
$\xi_2^3+\xi_0\xi_3^2+\xi_1^2\xi_4-2\xi_1\xi_2\xi_3-\xi_0\xi_2\xi_4$ over $\CCC$.
Therefore, if $m=n=4$, then $\CCC[T]^{G\times \slin(2,\CCC)}$ is generated by
$f_{22}^2-3f_{13}f_{31}+12f_{04}f_{40}$ and
$2f_{22}^3+27f_{04}f_{31}^2+27f_{13}^2f_{40}-9f_{13}f_{22}f_{31}-72f_{04}f_{22}f_{40}$
over $\CCC$.
\end{enumerate}
\end{example}

\section{Hyperdeterminant}
\mylabel{sec:hyper}

In this section, we make a brief remark on the relation
to hyperdeterminants defined by
Gelfand, Kapranov and Zelvinsky \cite{gkz1} (see also \cite{gkz2}).
First we recall the definition of the hyperdeterminant.

\begin{definition}[{\cite{gkz1}}]
\rm
Let $l_1$, $l_2$, \ldots, $l_d$ be  positive integers
and let $V$ be the image of the Segre embedding 
$\PPP_\CCC^{l_1}\times\cdots\times \PPP_\CCC^{l_d}
\to\PPP_\CCC^{(l_1+1)\cdots (l_d+1)-1}$.
If the projective dual variety $V^\vee$ is a hypersurface,
then the defining polynomial of $V^\vee$ is called the
hyperdeterminant of format $l_1$, \ldots, $l_d$.
\end{definition}

\begin{remark}
\rm
$V^\vee$ is a variety in the dual space $(\PPP_\CCC^{(l_1+1)\cdots(l_d+1)-1})^\ast$.
Therefore, if $V_1$, \ldots, $V_d$ are $\CCC$-vector spaces of dimension
$l_1+1$, \ldots, $l_d+1$ respectively, then $V^\vee$ is a subvariety
of $\PPP(V_1^\ast\otimes\cdots\otimes V_d^\ast)$.
In particular, if $V^\vee$ is a hypersurface, then the defining polynomial
of $V^\vee$ is an element of $\sym(V_1\otimes\cdots\otimes V_d)$.
Therefore, in our terminology, a hyperdeterminant of format
$l_1$, \ldots, $l_d$ is a polynomial of the entries of
$(l_1+1)\times \cdots\times (l_d+1)$-tensor of indeterminates.
\end{remark}

Now we recall some results of \cite{gkz1}.

\begin{thm}[{\cite[Theorem 1.3]{gkz1}}]
\mylabel{thm:gkz1.3}
The hyperdeterminant of format $l_1$, \ldots, $l_d$
exists if and only if 
$$
l_k\leq \sum_{j\neq k} l_j
$$
for any $k=1$, \ldots, $d$.
\end{thm}

\begin{prop}[{c.f. \cite[Proposition 1.4]{gkz1}}]
\mylabel{prop:gkz1.4}
The hyperdeterminant of format $l_1$, \ldots, $l_d$
is invariant under the action
of $\slin(l_1+1,\CCC)\times\cdots\times\slin(l_d+1,\CCC)$.
\end{prop}

In order to state the next proposition, we make the 
following:

\begin{definition}
\rm
Let $T$ be an $(l_1+1)\times\cdots\times(l_d+1)$-tensor
of indeterminates.
For a monomial $g$ of $K[T]$, we define the index support
$\isupp(g)$ of $g$ as follows:
$$
\isupp(g)\define\{(i_1,\ldots, i_d)\mid T_{i_1,\ldots, i_d}\in\supp(g)\}.
$$
\end{definition}

\begin{prop}[{c.f. \cite[Proposition 1.8]{gkz1}}]
\mylabel{prop:gkz1.8}
Let $P$ be a homogeneous element of $K[T]$
which is invariant under the action of
$\slin(l_1+1,\CCC)\times\cdots\times\slin(l_d+1,\CCC)$.
Then $P$ is divisible by the hyperdeterminant 
of format $l_1$, \ldots, $l_d$ if and only if
for any monomial $g$ of $P$ and any
$(j_1$, \ldots, $j_d)\in[l_1+1]\times\cdots\times[l_d+1]$,
there exists $(i_1$, \ldots, $i_d)\in\isupp(g)$ such that
$|\{k\mid i_k\neq j_k\}|\leq 1$.
\end{prop}

From now on, we consider the case where $d=3$, $l_1=l_2$ and
$l_3=1$ and we set $n=l_1+1$.
We recall the following result which is a special case of
\cite[Corollary 3.8]{gkz1}.

\begin{prop}
\mylabel{prop:gkz3.8}
The degree of the hyperdeterminant of format
$n-1$, $n-1$, $1$ is $2n(n-1)$.
\end{prop}

Since the hyperdeterminant of format $n-1$, $n-1$, $1$
is invariant under the action of
$\slin(n,\CCC)\times \slin(n,\CCC)\times\slin(2,\CCC)$,
it can be expressed as a polynomial of 
$f_{0,n}$, $f_{1,n-1}$, \ldots, $f_{n,0}$
by Theorem \ref{thm:m=n}.

Let $U_0$, $U_1$, \ldots, $U_n$ be indeterminates.
We introduce the degree reverse lexicographic order with
$U_n>U_{n-1}>\cdots>U_0$ on $\CCC[U_0, \ldots, U_n]$.
Note that for $g\in\CCC[U_0,\ldots, U_n]$ with
$g\neq 0$, 
$g(f_{0,n}$, $f_{1,n-1}$, \ldots, $f_{n,0})\neq 0$,
since $f_{0,n}$, \ldots, $f_{n,0}$ are 
algebraically independent over $\CCC$.

\begin{lemma}
\mylabel{lem:lm rel}
Let $a_0$, $a_1$, \ldots, $a_n$ and $b_0$, $b_1$, \ldots, $b_n$
be non-negative integers.
Then
$$
\lm(f_{0,n})^{a_0}\cdots\lm(f_{n,0})^{a_n}<
\lm(f_{0,n})^{b_0}\cdots\lm(f_{n,0})^{b_n}
$$
if and only if
$$
U_0^{a_0}\cdots U_n^{a_n}<
U_0^{b_0}\cdots U_n^{b_n}.
$$
In particular, for any $g\in\CCC[U_0$, \ldots, $U_n]$
with $g\neq 0$,
$$
\lm(g(f_{0,n},\ldots, f_{n,0}))=
\lm(g)(\lm(f_{0,n}),\ldots,\lm(f_{n,0})).
$$
\end{lemma}
\begin{proof}
Since $\lm(f_{k,n-k})=T_{111}\cdots T_{kk1}T_{k+1,k+1,2}\cdots T_{nn2}$,
the exponent of $T_{kk1}$ of
$\lm(f_{0,n})^{c_0}\cdots\lm(f_{n,0})^{c_n}
=c_k+c_{k+1}+\cdots+c_n$
for any non-negative integers $c_0$, $c_1$, \ldots, $c_n$.
Therefore,
$$
\lm(f_{0,n})^{a_0}\cdots\lm(f_{n,0})^{a_n}<
\lm(f_{0,n})^{b_0}\cdots\lm(f_{n,0})^{b_n}
$$
if and only if
there exists $h$ with $0\leq h\leq n$ such that
$\sum_{k=h}^n a_k<\sum_{k=h}^n b_k$ and
$\sum_{k=h'}^n a_k=\sum_{k=h'}^n b_k$ for any $h'$ with $0\leq h'<h$.
That is,
$$
U_0^{a_0}\cdots U_n^{a_n}<
U_0^{b_0}\cdots U_n^{b_n}.
$$
The last sentence of the lemma follows from the 
first part of the lemma.
\end{proof}

Now we state a corollary of Proposition \ref{prop:gkz1.8}.

\begin{cor}
\mylabel{cor:nec lm}
Let $\Phi(U_0$, $U_1$, \ldots, $U_n)$ be the element of
$\CCC[U_0$, \ldots, $U_n]$ such that 
$\Phi(f_{0,n}$, $f_{1,n-1}$, \ldots, $f_{n,0})$ is the 
hyperdeterminant of format $n-1$, $n-1$, $1$.
Then $\supp(\lm(\Phi))\cap\{U_0$, $U_1\}\neq \emptyset$.
\end{cor}
\begin{proof}
Assume the contrary.
Then 
since $\lm(f_{k,n-k})=T_{111}\cdots T_{kk1}T_{k+1,k+1,2}\cdots T_{nn2}$
for $k=2$, $3$, \ldots, $n$,
we see, by Lemma \ref{lem:lm rel}, that
\begin{eqnarray*}
&&\supp(\lm(\Phi(f_{0,n}, \ldots, f_{n,0})))\\
&=&\supp(\lm(\Phi)(\lm(f_{0,n}), \ldots, \lm(f_{n,0})))\\
&\subset &\{T_{111}, T_{221}, \ldots, T_{nn1},
T_{332}, T_{442}, \ldots, T_{nn2}\}.
\end{eqnarray*}
Therefore, we see that the pair
$\lm(\Phi(f_{0,n}, \ldots, f_{n,0}))$ and
$(1,2,2)\in[n]\times[n]\times[2]$ does not satisfy the condition
of Proposition \ref{prop:gkz1.8}.
This is a contradiction.
\end{proof}

\begin{example}
\rm
Consider the case where $n=4$.
Then the degree of the hyperdeterminant of format $3$, $3$, $1$
with respect to $T_{ijk}$ is 24 by Proposition \ref{prop:gkz3.8}.
Therefore, it is a polynomial of $f_{04}$, $f_{13}$, \ldots,
$f_{40}$ of degree 6.
By Example \ref{ex:inv} \ref{item:n=4},
$\slin(4,\CCC)\times \slin(4,\CCC)\times \slin(2,\CCC)$-invariant of 
degree 6 with respect to $f_{04}$, $f_{13}$, \ldots, $f_{40}$ is a
linear combination of
$(f_{22}^2-3f_{13}f_{31}+12f_{04}f_{40})^3$ and
$(2f_{22}^3+27f_{04}f_{31}^2+27f_{13}^2f_{40}-9f_{13}f_{22}f_{31}-72f_{04}f_{22}f_{40})^2$.
Therefore, we see that
$4(f_{22}^2-3f_{13}f_{31}+12f_{04}f_{40})^3-
(2f_{22}^3+27f_{04}f_{31}^2+27f_{13}^2f_{40}-9f_{13}f_{22}f_{31}-72f_{04}f_{22}f_{40})^2$
is the hyperdeterminant of format 3, 3, 1,
since by Corollary \ref{cor:nec lm}, term with $f_{22}^6$ must vanish.
\end{example}

\begin{remark}
\rm
If $n=m+1$, then
the determinant \refeq{eqn:inv det} is the
hyperdeterminant of format $m-1$, $m$, 1
\cite[Examples 4.9 and 4.10]{gkz1}.
\end{remark}

\end{document}